\documentclass[11pt]{amsart}

\usepackage{amssymb}
\usepackage{amsmath}
\usepackage{amsfonts}
\usepackage{graphicx}
\usepackage{amsthm}
\usepackage{enumerate}
\usepackage[mathscr]{eucal}
\usepackage{mathrsfs}
\usepackage{verbatim}
\usepackage{xcolor}
\usepackage{hyperref}
\usepackage{wrapfig}
\usepackage{color}

\makeatletter
\@namedef{subjclassname@2010}{%
  \textup{2010} Mathematics Subject Classification}
\makeatother

\numberwithin{equation}{section}

\theoremstyle{plain}
\newtheorem{theorem}{Theorem}[section]
\newtheorem{lemma}[theorem]{Lemma}
\newtheorem{proposition}[theorem]{Proposition}

\newtheorem{defn}[theorem]{Definition}

\theoremstyle{plain}

\numberwithin{equation}{section}

\theoremstyle{remark}
\newtheorem{remark}[theorem]{Remark}
\newtheorem{example}[theorem]{Example}

\DeclareMathOperator{\arccosh}{arccosh}
\DeclareMathOperator{\Tr}{Tr}

\begin{document}
\date{\today}

\title[Riemannian metrics with prescribed volume and finite spectrum]
{Riemannian metrics with prescribed volume and finite parts of Dirichlet spectrum}

\author{Xiang He, Zuoqin Wang}
\thanks{Partially supported by  National Key R and D Program of China 2020YFA0713100, and
by NSFC no. 12171446, 11721101.}
\address{School of Mathematical Sciences\\
University of Science and Technology of China\\
Hefei, 230026\\ P.R. China\\}
\email{hx1224@mail.ustc.edu.cn}

\address{School of Mathematical Sciences\\
University of Science and Technology of China\\
Hefei, 230026\\ P.R. China\\}
\email{wangzuoq@ustc.edu.cn}

\begin{abstract}
  In this paper we study the problem of prescribing Dirichlet eigenvalues on an arbitrary compact manifold $M$  of dimension $n\geq 3$ with a non-empty smooth boundary $\partial M$. We show that for any finite increasing sequence of real numbers $0<a_1<a_2 \leq a_3 \leq \cdots \leq a_N$ and any positive number $V$, there exists a Riemannian metric $g$ on $M$ such that $\mathrm{Vol}(M,g)=V$ and $\lambda^\mathcal{D}_k(M,g)=a_k$ for any integer $1 \leq k \leq N$.
\end{abstract}

\maketitle

\section{Introduction}

Let $M$ be a connected compact smooth manifold, either without boundary or with a smooth boundary $\partial M$. Associated with any Riemannian metric $g$ on $M$, there is a very important geometric operator, the Laplace-Beltrami operator $\Delta_g$, which lies in the center of differential geometry. The interplay between the spectral properties of $\Delta_g$  and  the geometric properties of the underlying space $(M, g)$ has received much attention in the past, especially after M. Kac raised his famous  question ``Can one hear the shape of a drum" in the seminal paper \cite{Kac}.

Among the many different eigenvalue problems associated with $\Delta_g$ that have been studied extensively in literature, three are most widely known, namely the closed eigenvalue problem when $M$ is closed, and the Neumann eigenvalue problem or the Dirichlet eigenvalue problem when $M$ has a non-empty boundary $\partial M$. In all these settings, the eigenvalues form a discrete increasing sequence of real numbers,
\[0 \le \lambda_1 < \lambda_2 \le \lambda_3 \le \cdots \le \lambda_n \le \cdots \to \infty,\]
where $\lambda_1=0$ for the closed and the Neumann eigenvalue problem (with eigenfunction the constant functions), and $\lambda_1>0$ for the Dirichlet eigenvalue problem. It is well-known that this sequence of real numbers encodes many important geometric/dynamical information of the background manifold, and one important goal in spectral geometry is to study the spectral map
\[
\mathcal S: \mathcal{M}(M) \to \mathbb R_{\ge 0, \nearrow}^\omega, \quad g \mapsto (\lambda_1, \lambda_2, \cdots)
\]
where $\mathcal{M}(M)$ is the space of all Riemannian metrics on $M$, and $\mathbb R_{\ge 0, \nearrow}^\omega$ is the set of all increasing sequence of non-negative real numbers that diverges to $\infty$.

According to the famous Weyl law,  large eigenvalues obey the growth rate
\[\lambda_k \sim  C(n, \mathrm{Vol}(M)) k^{2/n} \quad \text{as} \quad k \to \infty,\]
where $C(n, \mathrm{Vol}(M))$ is a constant that depends only on the dimension $n$ and the volume
$\mathrm{Vol}(M)$ of $(M, g)$. In particular, the spectral map $\mathcal S$ is far from  being surjective since it has a ``sparse tail part".
On the other hand,  a celebrated result of  Y. Colin de Verdiere (\cite{CV}, see also \cite{C-C} for an analogous  result on surfaces) tells us that for any closed manifold $M$ of dimension at least 3, and any given finite sequence of real numbers $0<a_2 \le \cdots  \le a_{N+1}$, there is a Riemannian metric $g$ on $M$ so that $\lambda_k(g)=a_k$. In other words,
if we look at the ``head part" of the spectral map $\mathcal S$, then the $N$-spectral map $\mathcal S^N$
 \[
\mathcal S^N: \mathcal{M}(M) \to \mathbb R_{> 0, \nearrow}^N, \quad g \mapsto (\lambda_2, \cdots, \lambda_{N+1}),
\]
of the closed eigenvalue problem is surjective for any $N \ge 1$. %
%
The Riemannian metric constructed in \cite{CV} may have a fairly small volume. However, by applying some kind of doubling surface trick via the crushed ice argument, J. Lohkamp showed in \cite{JL96} that not only one can prescribe the first $N$ closed eigenvalues, but also one can prescribe the volume and the total scalar curvature of $(M, g)$. For more results on prescribing finitely many eigenvalues, c.f. \cite{MD}, \cite{PJ}, \cite{PJ12} and \cite{PJ14} etc.

It is not hard to generalize the methods used in \cite{CV} directly to Neumann eigenvalues for manifolds with boundary, and obtain the same surjectivity result for Neumann eigenvalues. However, the case of Dirichlet eigenvalues is more complicated, especially since the existence of the first eigenvalue $\lambda_1$ which is positive and is simple. In fact, A. Hassanezhad, G. Kokarev and I. Polterovich proposed as an open problem (See Open Problem 4  in \cite{HKP}) whether there exists a Riemannian metric on a manifold of dimension at least 3 with boundary whose $k$th eigenvalue has prescribed multiplicity. In this paper we will give an affirmative  answer to this problem. In fact, we will show that on any manifold of dimension at least 3 with non-empty boundary, one can prescribe the first $N$ eigenvalues as well as the volume. Note that in contrast with the closed or Neumann eigenvalue case, for Dirichlet eigenvalues, the first eigenvalue can only has multiplicity 1. We also remark that the result does not conflict with the so-called universal inequalities for Dirichlet eigenvalues, since those inequalities depends on the geometry (especially the curvatures) of $(M, g)$.

To distinguish, we will use $\lambda_k^{\mathcal D}$ to represent the Dirichlet eigenvalues of $\Delta_g$.
The main result in this paper is
\begin{theorem}\label{mainthm}Let $M$ be a compact smooth manifold  of dimension $n\geq3$ with non-empty smooth boundary. Then for any given sequence $0<a_1<a_2\leq\cdots\leq a_N$ and any positive number $V$, there exists a Riemannian metric $g$ on $M$ such that
\[\mathrm{Vol}(M,g)=V\quad \text{and} \quad \lambda_k^\mathcal{D}(M,g)=a_k, \forall 1\leq k\leq N.\]
\end{theorem}

We briefly explain the idea of the proof. As in  \cite{CV} and  \cite{JL96},  the central idea of the proof is a stability condition that was first developed by V. Arnol'd in \cite{Arn} for the finite dimensional case. Roughly speaking, stability condition tells us that the eigenvalues are independent, so that one can perturb them to the values we want. More precisely, the theorem is proved in three steps. First we will construct a discrete graph with boundary whose Dirichlet eigenvalues (of the combinatorial Laplacian) are the given sequence of numbers. This can be done by studying the graph $G_N$ which is formed by adding $N$ boundary points to the complete graph of $N$ vertices. This step can be viewed as a modification of the corresponding construction in \cite{CV}. Then in the second step, we will construct a Riemann surface of very high genus and with lots of boundary components, and prove that there exists a Riemannian metric on it with prescribed first $N$ eigenvalues. In fact, the Riemann surface is constructed so that the graph associated with its thin-thick decomposition is exactly the discrete graph $G_N$.
In the third step, by embedding the surface constructed in Step two into the given manifold $M$, we are able to prove the existence of the desired Riemannian metric on $M$ (which is   concentrated near the embedded surface) with the prescribed eigenvalues by using the stability argument. Finally to prescribe the volume, we attach an elongated $n$-dimensional cuboid to the boundary of $M$ and apply the stability argument again. This last step of prescribing volume is new and is much simpler than the method used in \cite{JL96}, moreover it can be applied in the case of surfaces.

We remark that the same argument does not work for the case $n=2$, since  one can not embed a complete graph of $N$ vertices into the surface $\Sigma_{k,l}$ (the surface obtained by removing $l$ small disks from the closed surface $\Sigma_k$ of genus $k$) unless $k$ is large enough. In fact, the same conclusion  fail  for surfaces with boundary, because there is a topological restriction on the multiplicity of Dirichlet eigenvalues. For example, in \cite{Ber} A. Berdnikov showed that for any surface $M$ with $\chi(M)+b<0$, where $b$ is the number of boundary components of $M$, the $k$th Dirichlet eigenvalue has multiplicity no more than $2k-2(\chi(M)+b)+1$. In a forthcoming paper \cite{He} we will study the problem of prescribing Dirichlet eigenvalues of a surface with boundary, and by using a different graph, show that there is a Riemannian metric whose first $N$ Dirichlet eigenvalues are a given strictly increasing finite sequence of positive numbers. This can be viewed as an analogue of the corresponding result in \cite{CV} for  Dirichlet eigenvalues of surfaces with boundary.

The paper is arranged as follows. In Section 2 we will collect some results on spectral convergence that will be used later, and explain the notion of the stable metric. The proof of Theorem \ref{ct1} will be postponed to Section 6. Then in Section 3 we construct the discrete graph $G_N$, and in Section 4 we construct a Riemann surface with boundary modelled on $G_N$ with prescribed eigenvalues.  Finally in Section 5 we finish the proof of the main theorem.

\section{Spectral convergence and stable metrics}

In this section we list some backgrounds on spectral convergence and stable metrics that will be used later.

In studying the spectrum of the Laplace operator, it is more convenient to study the quadratic form associated to $\Delta$, namely the Dirichlet integral
\[
q(f) = \int_M |df|^2.
\]
It is well known that the Dirichlet eigenvalues of $\Delta$ are eigenvalues associated to $q$ on $H^1_0(M)$, while the Neumann eigenvalues (and the closed eigenvalues when $M$ is a closed manifold) are eigenvalues associated to $q$ on $H^1(M)$.

For any closed positive quadratic form $q$ with domain $D(q)$, there is a unique positive self-adjoint operator $A$ (whose domain is dense in $D(q)$) so that $q$ is the quadratic form of $A$ (c.f.  Theorem VIII.15 in \cite{RS}). Suppose $A$ has discrete spectrum $\{\lambda_i\}_{i=1}^\infty$ and suppose $\lambda_{N+1}>\lambda_N$, then the subspace of $D(q)$ generated by the first $N$ eigenfunctions of $A$ is well-defined, and is called the \textbf{$N$-eigenspace} of $q$ (resp. $A$).  Unless otherwise specified, all quadratic forms considered below will have discrete spectrum.

\subsection{Spectral convergence}\label{Aoe}

\noindent

For given integer $N>0$, let $E_0$ and $E_1$ be two subspaces of  dimension $N$ of a real Hilbert space $(\mathcal{H},\langle\ ,\ \rangle)$. Endow each $E_i$ with an inner product
 \begin{equation*}
 \langle x,y \rangle_i=\langle A_ix,y \rangle,\ x,y\in E_i, \qquad i=0,1,
 \end{equation*}
where $A_i: E_i \to E_i$ is a strictly positive operator. Moreover, suppose $E_1$ is \textbf{close} to $E_0$, in the sense that $E_1$ is the graph of a  bounded linear map $B\in\mathcal{L}(E_0,E_0^\perp)$. Consider an isometry
 \begin{equation}\label{iso}
 U_{E_0,E_1}=A_1^{-\frac{1}{2}}\mathcal{B}(\mathcal{B}^*\mathcal{B})^{-\frac{1}{2}}A_0^{\frac{1}{2}}
 \end{equation}
from $(E_0,\langle\ ,\ \rangle_0)$ to $(E_1,\langle\ ,\ \rangle_1)$, where $\mathcal{B}=I+B\in\mathcal{L}({E_0,E_1})$.
The following definition is taken from \cite{CV2}:
\begin{defn}\label{sd}
For positive quadratic forms  $q_i$ on $(E_i,\langle\ ,\ \rangle_i)$, $i=0,1$, if
 \[\|q_1\circ U_{E_0,E_1}-q_0\|_\infty \leq\varepsilon,\]
then we say $(E_0,\langle\ ,\ \rangle_0,q_0)$ and $(E_1,\langle\ ,\ \rangle_1,q_1)$ are $\varepsilon$-\textbf{close}. In particular, if $E_i$ is the $N$-eigenspace of some quadratic form $Q_i$, $q_i=Q_i|_{E_i}$, and if $(E_0,\langle\ ,\ \rangle_0,q_0)$ and $(E_1,\langle\ ,\ \rangle_1,q_1)$ are  $\varepsilon$-close, then we say $Q_0$ and $Q_1$ have an \textbf{$N$-spectral difference $\leq\varepsilon$}.
\end{defn}

A useful criterion for $\varepsilon$-closeness is given by Y. Colin de Verdi\`ere:
\begin{lemma}[\cite{CV2}, Critere I.3]\label{crit}
For any $\varepsilon>0$, there exists $M$ and $\alpha_i>0$ $(1\leq i\leq 5)$ depending on $\varepsilon$, such that if $\|q_1\|\leq M$, $\|A_0-I\|\leq\alpha_1$, $\|A_1-I\|\leq\alpha_2$, $\|B\|\leq\alpha_3$, $\max_{1\leq j\leq N}|\lambda_j(q_1)-\lambda_j(q_0)|\leq\alpha_4$, $q_1(x+Bx)\geq q_0(x)-\alpha_5|x|^2$, then $(E_0,\langle\ ,\ \rangle_0,q_0)$ and $(E_1,\langle\ ,\ \rangle_1,q_1)$ are $\varepsilon$-close.
\end{lemma}

\par Let $Q$ be a closed positive quadratic form on $\mathcal{H}$ with discrete spectrum, then we say $Q$ verifies hypothesis ($\ast$) if the eigenvalues of $Q$ satisfy the inequalities:
\[\label{hoe}
 \tag{$\ast$}\ \lambda_1\leq\cdots\leq\lambda_N\leq\lambda_N+\delta \le \lambda_{N+1}\leq M ,
\]
where $M,N,\delta$ will be fixed once and for all in what following.

\begin{theorem}\label{ct1}Let $Q$ be a positive quadratic form on Hilbert space $\mathcal{H}$ whose domain admits the $Q$-orthogonal decomposition $D(Q)=\mathcal{H}_0\oplus\mathcal{H}_\infty$. Suppose $Q_0=Q|_{\mathcal{H}_0}$ verifies the hypothesis (\ref{hoe}) and $\{\mu_i\}$ are the eigenvalues of $Q_0$.
{Let $\tilde \delta$ be the infimum of the gap between any two different eigenvalues from $\mu_1$ to $\mu_{N+1}$,}
then for any $\varepsilon>0$, there exists a constants $C>0$ (depends on $\tilde \delta,M,N,\varepsilon$) such that if $Q(x)\geq C|x|^2$ for all $x\in\mathcal{H}_\infty$, then $Q_0$ and $Q$ have a $N$-spectral difference $\leq\varepsilon$.
\end{theorem}

Theorem \ref{ct1} is a slightly weaker version of  Theorem I.7 in \cite{CV2} (where we assume $\tilde\delta$-dependence of the constant $C$).  The proof of Theorem \ref{ct1} will  be included in the last section, which is different from that of Theorem I.7 in \cite{CV2}. In fact,
the proof of Theorem I.7 in \cite{CV2} is based on Proposition I.5 in the same paper, for which we find a counterexample (see Example \ref{couExmp}). Fortunately this weaker version is enough for our purpose below.

We also need Remark 18 in \cite{PJ}, which is a variant of Theorem I.8 in \cite{CV2}:
\begin{theorem}[\cite{CV2}, \cite{PJ}]\label{ct2}
Let $Q$ be a positive quadratic form with domain $D(Q)$ and $\langle\ ,\ \rangle$ be an inner product on $D(Q)$. Suppose there is a sequence of inner products $\langle\ ,\ \rangle_n$ on $D(Q)$ and a sequence of quadratic forms $Q_n$ on $D(Q)$ satisfying
 \begin{enumerate}
   \item for any $x\in D(Q)$,
   \[\lim_{n\to\infty}|x|_n=|x|, \quad \lim_{n\to\infty}Q_n(x)=Q(x), \quad Q(x)\leq Q_n(x),\]
   \item there exists $C_1,C_2,\varepsilon_n>0$ with $\lim_{n\to\infty}\varepsilon_n=0$ such that \[C_1|x|\leq|x|_n\leq C_2|x|+\varepsilon_nQ(x)^{\frac{1}{2}}, \quad \forall x\in D(Q).\]
 \end{enumerate}
If $Q$ satisfies the hypothesis (\ref{hoe}), then for any $\varepsilon>0$, there exists integer $K$ depending on $\delta,M,N,\varepsilon$ such that $Q$ and $Q_n$ have an $N$-spectral difference $\leq\varepsilon$ for all $n>K$.
\end{theorem}

In most applications we use the $L^2$-inner product on Riemannian manifold and the associated quadratic form. More precisely, for any  compact Riemannian manifold  $(M,g)$, with or without boundary, we let
\[Q_g(f)=\int_M|\nabla_{g} f|_{g}^2\mathrm{d}V_{g}\]
be the quadratic form on $H^1(M, g)$ when $M$ has no boundary, or on $H_0^1(M, g)$ when $M$ has boundary.

In the following proofs, we will need to justify whether a function space is close enough to a certain $N$-eigenspace. Let $q$ be a positive quadratic form on $(\mathcal{H},\langle\ ,\ \rangle)$ with domain $\mathcal{D}$. Suppose $E_0$ and $E_1$ are subspaces of the same dimension of  $\mathcal{D}$, both equipped with the induced inner product. Denote
 \begin{equation}\label{toq}
 q_{E_1/E_0}=q|_{E_1}\circ U_{E_0,E_1},
 \end{equation}
where  $U_{E_0,E_1}$ is defined in (\ref{iso}).
Let $|x|_q=\big(|x|^2+q(x)\big)^{\frac{1}{2}}$ and $A$ be the self-adjoint operator associated with $q$. For any interval $I\subset\mathbb{R}$, denote $P_{q,I}$ be the spectral projection of $A$ on $I$. We need the following lemma for small eigenvalues, proved by B. Colbois and Y. Colin de Verdi{\`e}re in \cite{C-C}:

\begin{lemma}[\cite{C-C} Lemma I.1]\label{sl}
Fix $C_1>0$ and $N \in \mathbb N$. Suppose
 \begin{enumerate}
   \item $\dim(P_{q,[0,C_1]})\leq N$,
   \item There is an $N$ dimensional subspace $E_0$ of $\mathcal{D}$ and $\varepsilon>0$ such that
       \[|q(f,g)|\leq\varepsilon|f||g|_q, \qquad \forall f\in E_0, g\in\mathcal{D}.\]
 \end{enumerate}
Then there exists $\varepsilon_0(C_1,N)$, $C(C_1,N)>0$ such that if $\varepsilon\leq\varepsilon_0$, then
 \begin{enumerate}
   \item[$\mathrm{a)}$] If we denote the first $N$ eigenvalues of $A$ and $q|_{E_0}$ by $\lambda_1, \cdots, \lambda_N$ and $\mu_1,\cdots,\mu_N$ respectively, then $\lambda_N\leq C_1$, and   \[\mu_i-C\varepsilon^2\leq\lambda_i\leq\mu_i, \quad \forall 1\leq i\leq N.\]
   \item[$\mathrm{b)}$] For the subspace $E$ generated by the first $N$ eigenfunctions of $A$,
        \begin{equation*}
        \|q_{E/E_0}-q_{E_0}\|_\infty\leq C\varepsilon^{1+\frac{1}{N}}.
        \end{equation*}
 \end{enumerate}
\end{lemma}

\subsection{Stable metrics}

We will frequently use the isomorphism constructed in (\ref{iso}). In order to simplify the description, we introduce the conception of stable metrics following \cite{JL96}.

\par Let $M,\ M'$ be two compact manifolds with piecewise smooth boundaries and $B$ be a closed ball in $\mathbb{R}^m$. Suppose we have two continuous maps
\[f:B\to\mathcal{M}(M),\ \ F:B\to\mathcal{M}(M')\]
where $\mathcal{M}(M)$ and $\mathcal{M}(M')$ are the space of smooth Riemannian metrics on $M$ and $M'$ respectively. For any $g\in\mathcal{M}(M)$, let $E_N(g)$ be the $N$-eigenspace of the Laplacian of $g$ with certain boundary conditions and $\mathcal{Q}_N(g)$ be the space of quadratic forms on $E_N(g)$. For a fixed Riemannian metric $g_0$ on $M$, we can construct an $L^2$-isometry
\[i_p: (E_N(g_0), \langle , \rangle_{g_0})\to (E_N(f(p)), \langle , \rangle_{f(p)})\]
by (\ref{iso}) which is continuous on $p\in B$. In the following proofs, we will also have an injective map
\[ U_p:H^1(M,f(p))\to H^1(M',F(p))\]
using which we can pull-back the inner product on $U_p(E_N(f(p)))$ to $E_N(f(p))$. By composing with  (\ref{iso}), we may construct an $L^2$-isometry
\[I_p^1:(E_N(f(p)), \langle , \rangle_{f(p)}) \stackrel{(\ref{iso})}{\to} (E_N(f(p)), U_p^*(\langle , \rangle_{F(p)})) \to  (U_p(E_N(f(p))), \langle , \rangle_{F(p)}),\]
Also (\ref{iso}) gives another $L^2$-isometry
\[I_p^2:(U_p(E_N(f(p))), \langle , \rangle_{F(p)})\to (E_N(F(p)), \langle , \rangle_{F(p)}).\]
Composing these three isometries, we get an $L^2$-isometry
\[I_p = I_p^2 \circ I_p^1 \circ i_p: (E_N(g_0), \langle , \rangle_{g_0}) \to (E_N(F(p)), \langle , \rangle_{F(p)}).\]

Using $i_p$ and $I_p$, one can transform the Dirichlet integrals on $E_N(f(p))$ and $E_N(F(p))$ to quadratic forms on $E_N(g_0)$, to get the following two continuous maps:
\[\Phi(f):B\to\mathcal{Q}_N(g_0),\ \Phi(f)(\varphi)=\int_M |\nabla_{f(p)}i_p(\varphi)|^2\mathrm{d}V_{f(p)},\]
\[\Phi(F):B\to\mathcal{Q}_N(g_0),\ \Phi(F)(\varphi)=\int_{M'} |\nabla_{F(p)} I_p(\varphi)|^2\mathrm{d}V_{F(p)}.\]
For simplicity, let $0$ be the center of $B$. Now we write down the definition of stable metrics:
\begin{defn}[\cite{JL96} Definition 2.1]
Let $f:B\to\mathcal{M}(M)$ be a continuous family of metrics  on $M$ with $f(0)=g_0$.
 If there exists an $\varepsilon=\varepsilon(f,g_0)>0$ such that
for any continuous  family of metrics $F:B\to\mathcal{M}(M')$ on $M'$ with \[\|\Phi(F)-\Phi(f)\|_{C^0(B)}<\varepsilon,\] there exists a point $p\in\mathrm{int}(B)$ with
 \[\Phi(F)(p)=\Phi(f)(0).\]
Then we say $f$ is  a \textbf{stable} family  around $g_0$, $g_0$ is a \textbf{stable metric}, and $F$ is \textbf{spectrally near} to $f$.
 \end{defn}

\section{Constructing graph with prescribed Dirichlet eigenvalues}\label{cog}

For any positive integer $N$, consider the finite graph $G_N=(V,E)$ which is obtained from the complete graph with $N$ vertices (as interior vertices) by adding $N$ boundary vertices, one to each interior vertex. More precisely,  the vertex set of $G_N$ is $V=\{v_1,\dots,v_N\}\cup\{u_1,\dots,u_N\}$, and the edge set of $G_N$ is
 \[E=\{(v_i,v_j),1\leq i<j\leq N\}\cup\{(v_k,u_k),1\leq k\leq N\}.\]

Note that we used different letters to distinguish the interior vertices and the boundary vertices. Let
 \begin{equation}\label{H}
 H=\{f:V\to\mathbb{R}|\ f(u_i)=0,\ \forall1\leq i\leq N\}.
 \end{equation}
For a given measure $\mu=\sum_{i=1}^N\mu_i\delta (v_i)$ ($\mu_i>0$) on $G_N$, we can define an inner product $\langle\ ,\ \rangle_\mu$ on $H$ via
 \begin{equation}\label{Gip}
 \langle f,g\rangle_\mu=\sum_{i=1}^N \mu_if(v_i)g(v_i),\ \forall f,g\in H.
 \end{equation}
For a given edge weight function
\[\Theta:E\to\mathbb{R}_{>0},\ (v_i,v_j)\mapsto \theta_{ij},\ (v_i,u_i)\mapsto\theta_i\]
(which will be viewed as a vector in $\mathbb{R}_{>0}^{|E|}$ in the following discussion),   define a quadratic form $q_\Theta$ on $H$,
 \begin{equation}
 q_\Theta(f)=\sum_{1\leq i<j\leq N}\theta_{ij}(f(v_i)-f(v_j))^2+\sum_{i=1}^N\theta_i f(v_i)^2.
 \end{equation}
Let $\Delta_\Theta$ be the combinatorial Dirichlet Laplacian associated with $(H,\mu,q_\Theta)$, namely
 \begin{equation}
 \Delta_\Theta(f)(v_i)=\frac{1}{\mu_i}\sum_{j\neq i}\theta_{ij}(f(v_i)-f(v_j))+\frac{\theta_i}{\mu_i}f(v_i)
 \end{equation}
where we set $\theta_{ij}=\theta_{ji}$ when $i>j$.

We are interested in the eigenvalues of $\Delta_\Theta$, which will be denoted as
 \begin{equation}
 0<\lambda_1^\Theta<\lambda_2^\Theta\leq\cdots\leq\lambda_N^\Theta.
 \end{equation}
It turns out that for any fixed $\mu$, one can prescribe these eigenvalues by choosing suitable edge wight function $\Theta$:
\begin{theorem}\label{CoG}
Fix any measure $\mu=\sum_{i=1}^N\mu_i\delta(v_i)$ on $G_N$. For any sequence $0<\lambda_1<\lambda_2\leq\cdots\leq\lambda_N$, there exists a function $\Theta:E\to\mathbb{R}_{>0}$ such that the eigenvalues of $\Delta_\Theta$ are precisely $\lambda_i$'s.
\end{theorem}

\begin{proof} Let $\Gamma_N$ be the subgraph of $G_N$ with vertices $\{v_1,\cdots,v_N\}$. By Theorem 2 in \cite{CV1}, there exists a function
 \begin{equation*}
 \Theta':\{(v_i,v_j),1\leq i<j\leq N\}\to\mathbb{R}_{>0},\ (v_i,v_j)\mapsto \theta'_{ij}
 \end{equation*}
such that the associated combinatorial Laplacian $\Delta_{\Theta'}$ on $\Gamma_N$ defined by
 \begin{equation*}
 \Delta_{\Theta'}(f)(v_i)=\frac{1}{\mu_i}\sum_{j\neq i}\theta'_{ij}(f(v_i)-f(v_j))
 \end{equation*}
has eigenvalues $\{0,\lambda_2-\lambda_1,\cdots,\lambda_N-\lambda_1\}$. One can check that  \[\Theta:E\to\mathbb{R}_{>0},  \quad \theta_{ij}=\theta'_{ij},\; \theta_i=\lambda_1\mu_i\]
is  what we want.
\end{proof}

\section{Constructing surface with prescribed Dirichlet spectrum}\label{sec4}

In this section, we will prove that for any increasing sequence
\[0<a_1<a_2\leq a_3\leq\cdots\leq a_N,\]
there exists a metric $g$ on the surface
\[X_N=\Sigma_{\frac{N(N-3)}{2}+1}-\{D_1,\cdots,D_N\}\]
(i.e. the surface of genus $\frac{N(N-3)}{2}+1$ with $N$ punctures) such that the first $N$ Dirichlet eigenvalues of $(X_N, g)$ are $a_i$'s.

\subsection{Capacitors}

First we recall the definition of capacitor from \cite{C-C} and do some calculations that will be used later.

\par Let $\overline{X}$ be a compact Riemannian manifold of dimension $n$, $X$ be its interior. Suppose we have a partition of the boundary $\partial X=C_+\cup C_-$ where $C_+$ and $C_-$ are two disjoint manifolds of dimensional $n-1$. Consider the \textbf{capacitor} $\mathcal{C}=(X,C_+,C_-)$.
The \textbf{capacity} $\mathrm{cap}(\mathcal{C})$ of the capacitor $\mathcal{C}=(X,C_+,C_-)$ is defined to be
 \begin{equation*}
 \mathrm{cap}(\mathcal{C}):=\inf\{\int_X|\nabla f|^2:\ f\in H^1(\overline{X})\ f|_{C_+}=1\ f|_{C_-}=0\}.
 \end{equation*}
Then there exists a unique function $f_0\in H^1(\overline{X})$ with $f_0|_{C^+}=1\ f_0|_{C^-}=0$ such that
 \begin{equation*}
 \mathrm{cap}(\mathcal{C})=\int_X|\nabla f_0|^2.
 \end{equation*}
The function  $f_0$ is known as the \textbf{equilibrium potential} of the capacitor $\mathcal{C}$.
For any $a, b \in \mathbb R$, we will denote
\[f_{ab}=(a-b)f_0+b.\]

For any fixed $l>0$ and any $\varepsilon>0$, let
\[a^\varepsilon=\arccosh(\frac{l}{\pi\varepsilon}).\]
In \cite{C-C}, B. Colbois and Y. Colin de Verdi\`{e}re studied the capacity of the  capacitor $\mathcal{C^\varepsilon}=(Y^\varepsilon,C^\varepsilon_+,C^\varepsilon_-)$, where   \[Y^\varepsilon=[-a^\varepsilon,a^\varepsilon]\times\mathbb{R}/\mathbb{Z}\]
is a cylinder equipped  with hyperbolic metric
 \[g^\varepsilon=\mathrm{d}x^2+(\pi\varepsilon\cosh x)^2\mathrm{d}\theta^2,\]
and $C_{\pm}^\varepsilon=\{\pm a^\varepsilon\}\times\mathbb{R}/\mathbb{Z}$. They proved
\begin{proposition}[\label{CC3.3}\cite{C-C}, Prop \uppercase\expandafter{\romannumeral3}.3]
For fixed $l>0$ and any $\varepsilon>0$, let $f_0^\varepsilon$ be the equilibrium potential of the capacitor $\mathcal{C^\varepsilon}=(Y^\varepsilon,C^\varepsilon_+,C^\varepsilon_-)$, then
 \begin{enumerate}[(1)]
   \item $\lim_{\varepsilon\to0}\frac{\mathrm{cap}(\mathcal{C^\varepsilon})}
       {\varepsilon}=1$,
   \item $\forall a,b\in\mathbb{R}$, $\lim_{\varepsilon\to0}\int_{Y^\varepsilon}|f^\varepsilon_{ab}|^2
       =l(a^2+b^2)$,
   \item For the unit outward normal vector field  $\nu$ along the boundary $\partial Y^\varepsilon$,
        \begin{equation*}
        \|\frac{\partial f_0^\varepsilon}{\partial \nu}\|_{L^2(C^\varepsilon_\pm)}=O(\varepsilon),\ \text{as}\ \varepsilon\to0,
        \end{equation*}
   \item The trace maps $\Tr^\varepsilon_{\pm}:H^1(Y^\varepsilon)\to L^2(C_\pm^\varepsilon)$ are uniformly bounded by norm for $\varepsilon$ close to 0.
 \end{enumerate}
\end{proposition}

To handle the Dirichlet boundary, we also need to study the analogous capacity for the ``half hyperbolic cylinder", namely
\[Z^\varepsilon=[-a^\varepsilon,0]\times\mathbb{R}/\mathbb{Z}\] equipped with the hyperbolic metric $g^\varepsilon$ given above:
\begin{proposition}\label{SCC3.3}
For fixed $l>0$ and any $\varepsilon>0$. Let $f_0^\varepsilon$ be the equilibrium potential of the capacitor $\mathcal{C}_1^\varepsilon=(Z^\varepsilon,C_+^\varepsilon,C_-^\varepsilon)$, where   $C^\varepsilon_+=\{-a^\varepsilon\}\times\mathbb{R}/\mathbb{Z}$, $C_-^\varepsilon=\{0\}\times\mathbb{R}/\mathbb{Z}$, then
 \begin{enumerate}[(1)]
   \item $\lim_{\varepsilon\to0}\frac{\mathrm{cap}(\mathcal{C}^\varepsilon
       _1)}{\varepsilon}=2$,
   \item $\forall a\in\mathbb{R}$, $\lim_{\varepsilon\to0}\int_{Z^\varepsilon}|f^\varepsilon_{a0}|^2
       =la^2$,
   \item For the unit outward normal vector field $\nu$ along the boundary $C_+^\varepsilon$,
        \begin{equation*}
        \|\frac{\partial f_0^\varepsilon}{\partial\nu}\|_{L^2(C_+^\varepsilon)}=O(
        \varepsilon),\ \text{as }\varepsilon\to0,
        \end{equation*}
   \item The trace map $\Tr^\varepsilon:H^1(Z^\varepsilon)\to L^2(C_+^\varepsilon)$ is uniformly bounded by norm for $\varepsilon$ close to 0.
 \end{enumerate}
\end{proposition}

\begin{proof}
The proof is very similar to that of Prop \uppercase\expandafter{\romannumeral3}.3 in \cite{C-C}. The only difference is that   the equilibrium potential $f_0^\varepsilon$ of the current capacitor $\mathcal{C}_1^\varepsilon$ solves the equation
\[\left\{\begin{array}{l}f''(x)+(\tanh x)f'(x)=0, \\
f(-a^\varepsilon)=1,  f(0)=0\end{array}\right.\]
and the solution is
\[f_0^\varepsilon=\delta^\varepsilon\arcsin(\tanh x),\]
where $\delta^\varepsilon=-(\arcsin(\tanh a^\varepsilon))^{-1}$ and $\lim_{\varepsilon\to0}\delta^\varepsilon=-\frac{2}{\pi}$. It follows
 \begin{equation*}
 \begin{aligned}
 \lim_{\varepsilon\to 0}\frac{\mathrm{cap}(\mathcal{C}_1^\varepsilon)}{\varepsilon}
 &= \lim_{\varepsilon\to0}\frac{1}{\varepsilon}\int_0^1
 \int_{-a^\varepsilon}^0\big(\delta^\varepsilon(\cosh x)^{-1}\big)^2\pi\varepsilon\cosh x\mathrm{d}x\mathrm{d}\theta\\
 &= \lim_{\varepsilon\to0}(\delta^\varepsilon)^2\pi\arcsin(\tanh a^\varepsilon)=2,
 \end{aligned}
 \end{equation*}
which proves (1). The proofs of (2), (3) and (4) are similar to those  in \cite{C-C}  and will be omitted.
\end{proof}

\subsection{Existence of the desired metric}\label{CoS}

Next we   prove the existence of a desired metric on the surface $X_N=\Sigma_{\frac{N(N-3)}{2}+1}-\{D_1,\cdots,D_N\}$.

Firstly, we fix the measure
\[\mu_0=\sum_{i=1}^N2\pi(N-2)\delta(v_i)\] on the graph $G_N=(V,E)$ described in Section \ref{cog}. For any edge weight function $\Theta:E\to\mathbb{R}_{>0}$, $(v_i,v_j)\mapsto\theta_{ij}$, $(v_i,u_i)\mapsto\theta_i$, we shall construct a family of Riemannian metrics $\{g^{\varepsilon}_{\Theta}\}_{\varepsilon>0}$ with constant curvature $ -1$ on $X_N$. For simplicity we denote $(X_N, g^{\varepsilon}_{\Theta})$ by $X^{\varepsilon}_{\Theta}$.

For each $1\leq i\leq N$, let $X^\varepsilon_{\Theta;i}$ be the surface constructed by gluing $N-2$ hyperbolic pants shown in the graph below:\\
\includegraphics[scale=.7]{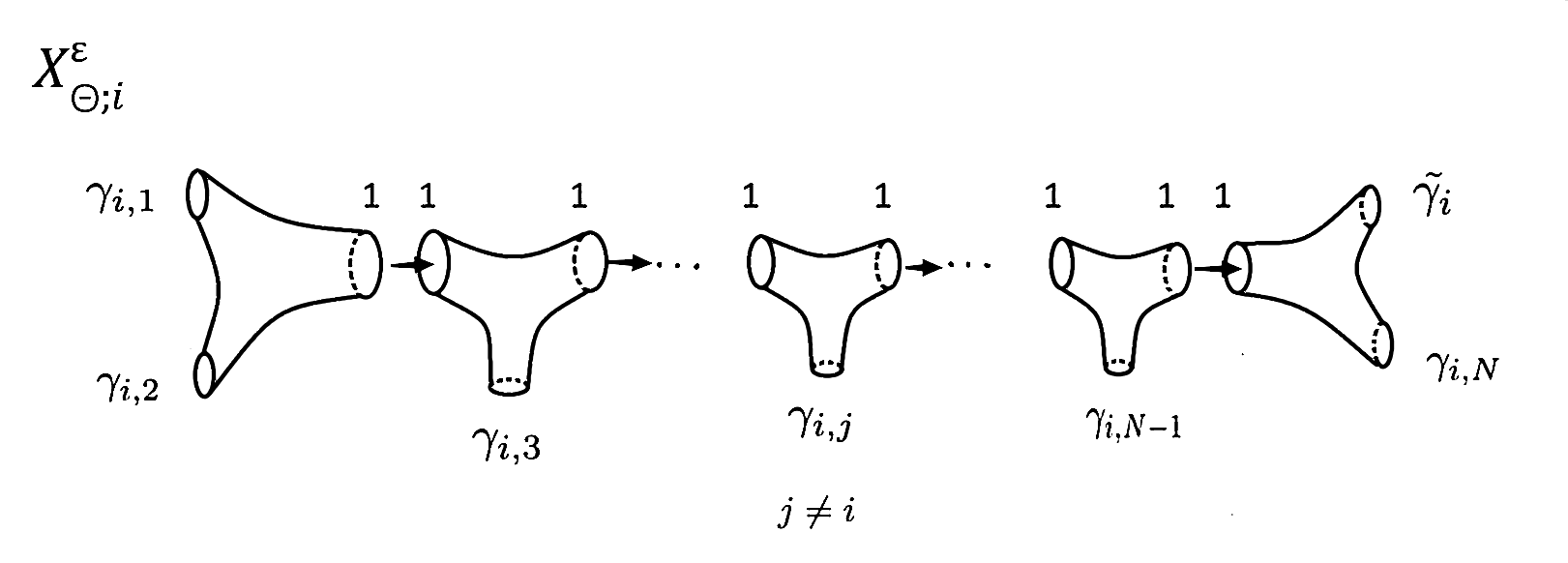}\\
\noindent
More precisely, we denote $P_{\gamma_1,\gamma_2,\gamma_3}$ to be a hyperbolic pant whose boundary geodesics are $\gamma_1, \gamma_2$ and $\gamma_3$. In our construction, these $\gamma_i$'s will be one of the following three cases,
\begin{itemize}
  \item   a geodesic of length $\pi\varepsilon\theta_{ij}$, which will be denoted by $\gamma_{i,j}$,
  \item   a geodesic of length 1, which will simply denoted by 1
  \item   a  geodesic of length $\frac{\pi\varepsilon}{2}\theta_i$, which will be denoted by $\tilde{\gamma_i}$.
\end{itemize}
We glue such pants as follows:
\begin{itemize}
  \item the pant $P_{\gamma_{i,1},\gamma_{i,2},1}$ is glued to the pant $P_{1,\gamma_{i,3},1}$ along the adjacent geodesic of length 1 without twist,
  \item the pant $P_{1,\gamma_{i,j-1},1}$ is glued to the pant $P_{1,\gamma_{i,j},1}$ along the adjacent geodesic of length 1 without twist,
  \item the pant $P_{1,\gamma_{i,N-1},1}$ is glued to the pant $P_{1,\gamma_{i,N},\tilde{\gamma_i}}$ along the adjacent geodesic of length 1 without twist.
\end{itemize}
The resulting surface is denoted by $X^\varepsilon_{\Theta;i}$.

Finally we glue all these $X^\varepsilon_{\Theta;i}$, $1 \le i \le N$, to get the desired surface $X^{\varepsilon}_{\Theta}$: the surfaces $X_{\Theta;i}^\varepsilon$ and $X_{\Theta;j}^\varepsilon$ are glued along geodesics $\gamma_{i,j}$ and $\gamma_{j,i}$ without twist. Topologically the  surface $X^{\varepsilon}_{\Theta}$ is $X_N$, with boundary $\tilde{\gamma_i}$, $1 \le i \le N$. We refer to \cite{PB} for the detail of such gluing constructions. In particular, there exist  a tubular neighborhood $Y^\varepsilon_{ij}$ of $\gamma_{i,j}$ in $X^{\varepsilon}_{\Theta}$ of width
\[a^\varepsilon_{ij}=\arccosh (\frac{1}{\pi\varepsilon\theta_{ij}}),\]
i.e. $Y^\varepsilon_{ij}=[-a^\varepsilon_{ij},a^\varepsilon
_{ij}]\times\mathbb{R}/\mathbb{Z}$ with metric
\[\mathrm{d}x^2+(\pi\varepsilon\theta_{ij}\cosh x)^2\mathrm{d}\theta^2\]
and a collar neighborhood $Z_i^\varepsilon$ of $\tilde{\gamma_i}$ in $X^{\varepsilon}_{\Theta}$ of width \[b_i^\varepsilon=\arccosh(\frac{2}{\pi\varepsilon\theta_i}),\]
i.e. $Z^\varepsilon_i=[-b^\varepsilon_i,0]\times\mathbb{R}
/\mathbb{Z}$ with metric
\[\mathrm{d}x^2+(\frac{\pi\varepsilon}{2}\theta_i\cosh x)^2\mathrm{d}\theta^2.\]
Remove  these cylinder regions from $X_{\Theta;i}^\varepsilon$ and denote the remaining part by $\widehat{X^\varepsilon_{\Theta;i}}$, namely
\[\widehat{X^\varepsilon_{\Theta;i}}=X_{\Theta;i}^\varepsilon\setminus \big((\cup_{j\neq i}Y^\varepsilon_{ij})\cup Z^\varepsilon_i\big),\]
and let $E_0^\varepsilon$ be the space of functions $f\in H_0^1(X^{\varepsilon}_{\Theta})$ such that $f|_{\widehat{X^\varepsilon_{\Theta;i}}}$ is constant and $f$ is harmonic on $Y^\varepsilon_{ij}$ and $Z^\varepsilon_i$.  
 %
By Proposition \ref{CC3.3} and \ref{SCC3.3} with $l=1$, it is easy to get
\begin{proposition}\label{cp}
Fix $x_1, \cdots, x_N$. For any $f_\varepsilon\in E_0^\varepsilon$ with $f_\varepsilon|_{\widehat{X_{\Theta;i}^\varepsilon}}=x_i$, we have
 \begin{equation}
 \lim_{\varepsilon\to0}\frac{\int_{X^\varepsilon_\Theta}f_\varepsilon^2}
 {\sum_{i=1}^N2\pi(N-2)x_i^2}=1
 \end{equation}
and
 \begin{equation}
 \lim_{\varepsilon\to0}\frac{\int_{X_\Theta^\varepsilon}|
 \nabla f_\varepsilon|^2}{\varepsilon\cdot(\sum_{1\leq i<j\leq N}\theta_{ij}(x_i-x_j)^2+\sum_{i=1}^N\theta_ix_i^2)}=1
 \end{equation}
\end{proposition}

As in \cite{C-C}, we need to show that the Dirichlet Laplacian $\Delta$ on $X^{\varepsilon}_{\Theta}$ satisfies the conditions (1) and (2) in Lemma \ref{sl}. As usual we let $P_{\Delta,I}$ be the spectral projection of $\Delta$ on interval $I$.
\begin{proposition}\label{ps}
Consider the Dirichlet Laplacian $\Delta$ on the surface $X^{\varepsilon}_{\Theta}$. Then for $\varepsilon$ close to 0, there exist  uniform constants $C_1$, $C_2>0$ such that
 \begin{enumerate}
   \item $\forall f\in E_0^\varepsilon,\ g\in H^1_0(X^{\varepsilon}_{\Theta})$,
   \begin{equation*}
   \big|\int_{X^{\varepsilon}_{\Theta}}\nabla f\cdot\nabla g\big|\leq C_1\cdot\varepsilon\|f\|_{L^2(X^{\varepsilon}_{\Theta})}\|g\|
       _{H^1_0(X^{\varepsilon}_{\Theta})},
   \end{equation*}
   \item $\dim(P_{\Delta,[0,C_2]})=N$.
 \end{enumerate}
\end{proposition}

\begin{proof}
The proof is simply an adjustment of the proof of Proposition II.1 in \cite{C-C} to the case of Dirichlet Laplacian. For (1),
 \begin{equation*}
 \begin{aligned}
 \big|\int_{X^{\varepsilon}_{\Theta}}\nabla f\cdot\nabla g\big|=&\big|\sum_{1\leq i<j\leq N}\int_{Y^\varepsilon_{ij}}\nabla f\cdot\nabla g+\sum_{i=1}^N\int_{Z^\varepsilon_i}\nabla f\cdot\nabla g\big|\\
 =&\big|\sum_{1\leq i<j\leq N}\int_{\partial Y^\varepsilon_{ij}}\frac{\partial f}{\partial\nu}\cdot g+\sum_{i=1}^N\int_{\{-b_i^\varepsilon\}\times\mathbb{R}
 /\mathbb{Z}}\frac{\partial f}{\partial\nu}\cdot g\big|\\
 \leq& \sum_{1\leq i<j\leq N}(\int_{\partial Y^\varepsilon_{ij}}|\frac{\partial f}{\partial\nu}|^2)^{\frac{1}{2}}(\int_{\partial Y^\varepsilon_{ij}}g^2)^{\frac{1}{2}} \\
 & \indent  +\sum_{i=1}^N(
 \int_{\{-b_i^\varepsilon\}\times\mathbb{R}/\mathbb{Z}}
 |\frac{\partial f}{\partial\nu}|^2)^{\frac{1}{2}}
  (\int_{\{-b_i^\varepsilon\}\times\mathbb{R}/\mathbb{Z}}
 g^2)^{\frac{1}{2}},
 \end{aligned}
 \end{equation*}
and thus by Proposition \ref{CC3.3} and \ref{SCC3.3}, there exists  $C_1>0$ such that
 \begin{equation*}
 |\int_{X^{\varepsilon}_{\Theta}}\nabla f\cdot\nabla g|\leq C_1\cdot\varepsilon\|f\|_{L^2(X^{\varepsilon}_{\Theta})}\|g\|_{H^
 1_0(X^{\varepsilon}_{\Theta})}.
 \end{equation*}

For (2), we can't apply Theorem 1.1 in \cite{DPRS} directly, since $X^\varepsilon_\Theta$ is not a closed surface. However, we may first ``close up" $X^\varepsilon_\Theta$ as follows: for each $1 \le i \le N$, we glue a pant $P_{\tilde{\gamma_i},1,1}$ to $X^\varepsilon_i$ along the geodesic $\tilde{\gamma_i}$ without twist, and glue the two geodesics of length 1 of $P_{\tilde{\gamma_i},1,1}$ together without twist. The resulting closed surfaces is denoted by $\widetilde{X^{\varepsilon}_{\Theta}}$. We  will  denote by $P_{\tilde{\gamma_i}}$   the surface constructed by gluing the two geodesics of length 1 of $P_{\tilde{\gamma_i},1,1}$ together without twist.\\
\includegraphics[scale=0.3]{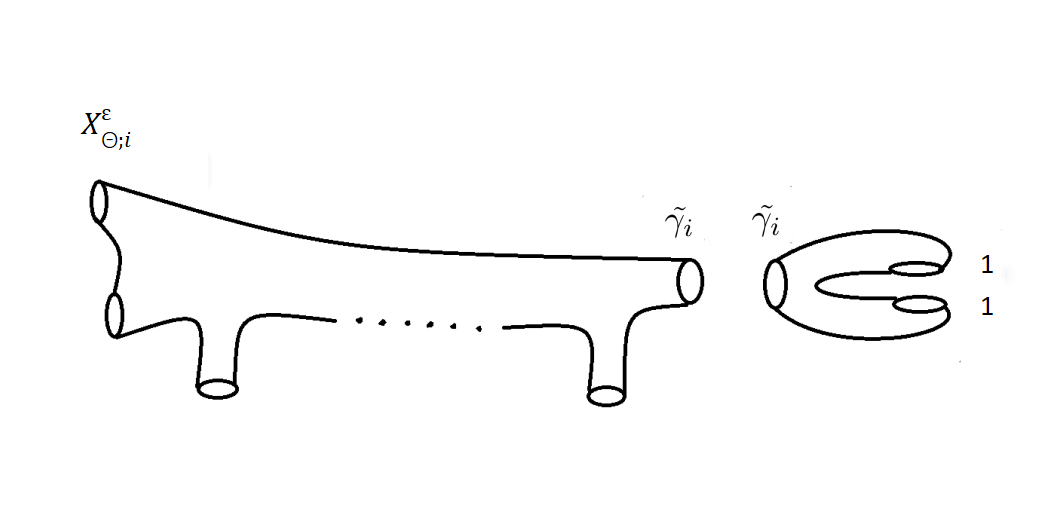}

Now we may apply Theorem 1.1 in \cite{DPRS} to conclude that there exists a constant  $C_2>0$, independent of $\varepsilon$, such that for all $\varepsilon$ close to 0,
\[\lambda^\varepsilon_{2N+1}>C_2,\]
where $\lambda^\varepsilon_{2N+1}$ is the $(2N+1)$th eigenvalue of the closed Laplacian on $\widetilde{X^{\varepsilon}_{\Theta}}$.

 Let $\{\varphi^\varepsilon_i\}_{i=1}^{N+1}$ be the first $N+1$ orthonormal eigenfunctions of the Dirichlet Laplacian $\Delta$ on the surface $X^{\varepsilon}_{\Theta}$,   viewed as functions in $H^1(\widetilde{X^{\varepsilon}_{\Theta}})$ by  zero extension on each $P_{\tilde{\gamma_i}}$. For $1\leq i\leq N$, let
 \[W^\varepsilon_i=[-b^\varepsilon_i,0]\times\mathbb{R}/\mathbb{Z}\] (where $b^\varepsilon_i=\arccosh(\frac{2}{\pi\varepsilon
\theta_i})$) be the collar neighborhood of $\tilde{\gamma_i}$ in $P_{\tilde{\gamma_i}}$, with hyperbolic metric $\mathrm{d}x^2+(\frac{\pi\varepsilon}{2}\theta_i\cosh x)^2\mathrm{d}\theta^2$, and let $\psi_i^\varepsilon$ be the function in $H_0^1(P_{\tilde{\gamma_i}})$ which equals 1 on $P_{\tilde{\gamma_i}}\setminus W_i^\varepsilon$ and is harmonic on $W_i^\varepsilon$. We   also view $\psi_i^\varepsilon$  as a function in $H^1(\widetilde{X^{\varepsilon}_{\Theta}})$ by   zero extension on $\widetilde{X^{\varepsilon}_{\Theta}}\setminus P_{\tilde{\gamma_i}}$. By Proposition \ref{SCC3.3},
 \begin{equation*}
 \lim_{\varepsilon\to0}\frac{\int_{\widetilde{X^{\varepsilon}_{\Theta}}}
 |\nabla\psi_i^\varepsilon|^2}{\int_{\widetilde{X^{\varepsilon}_{\Theta}}}
 (\psi_i^\varepsilon)^2}=0.
 \end{equation*}
Obviously $\{\varphi_i^\varepsilon,\psi_j^\varepsilon\}_{1\leq i\leq N+1,1\leq j\leq N}$ are pairwise orthogonal in both $L^2(\widetilde{X^{\varepsilon}_{\Theta}})$ and $H^1(\widetilde{X^{\varepsilon}_{\Theta}})$. So by min-max principle,
\[\lambda_{N+1}(\Delta)\geq\lambda^\varepsilon_{2N+1}>C_2.\]
On the other hand, by Proposition \ref{CC3.3} and \ref{SCC3.3}, we have
\[\lim_{\varepsilon\to0}\lambda_i(\Delta)=0 \]
for $1\leq i\leq N$. Thus $\dim(P_{\Delta,[0,C_2]})=N$.
\end{proof}

Recall that in section \ref{cog}, we   defined the function space
\[H=\{f:V\to\mathbb{R}|\ f(u_i)=0,\ \forall 1\leq i\leq N\}\]
and  the quadratic form $q_\Theta$ on $(H,\langle\ ,\ \rangle_{\mu_0})$,
 \begin{equation}\label{qtheta}
 q_\Theta(f)=\sum_{1\leq i<j\leq N}\theta_{ij}(f(v_i)-f(v_j))^2+\sum_{i=1}^N\theta_if(v_i)^2.
 \end{equation}
Let $q^\varepsilon_\Theta$ be the quadratic form on $H^1_0(X^\varepsilon_\Theta)$ defined by
 \begin{equation}\label{qetheta}
 q^\varepsilon_\Theta(\varphi)=\int_{X^\varepsilon_\Theta}|
 \nabla\varphi|^2.
 \end{equation}
To compare $q_\Theta$ and $q^\varepsilon_\Theta$, we consider the map
 \begin{equation}
 U^\varepsilon_\Theta:H\to H_0^1(X^\varepsilon_\Theta),\ f\mapsto U^\varepsilon_\Theta f
 \end{equation}
where $U^\varepsilon_\Theta f \in H_0^1(X^\varepsilon_\Theta)$ is the function that takes the value $f(v_i)$ on $\widehat{X_{\Theta,i}^\varepsilon}$ and is harmonic on $Y^\varepsilon_{ij}$ and $Z^\varepsilon_i$.

Let $E^\varepsilon$ be the $N$-eigenspace  of $q^\varepsilon_\Theta$. For the inner product \[(U^\varepsilon_\Theta)^\ast(\langle\cdot,\cdot\rangle_{E_0^\varepsilon}
):=\langle U^\varepsilon_\Theta\cdot,U^\varepsilon_\Theta\cdot\rangle_{E_0^\varepsilon}\] on $H$,   there exists an isometry $\big($see (\ref{iso})$\big)$
\[U_{H,H}: (H,\langle\ ,\ \rangle_{\mu_0}) \to  \big(H,(U^\varepsilon)^\ast(\langle\cdot,\cdot\rangle_{E^\varepsilon
_0})\big).\]
Consider the quadratic form
\[{q^\varepsilon_{\Theta;1}}=(q^\varepsilon_\Theta)_{E^\varepsilon/E_0^\varepsilon}
\circ U^\varepsilon_\Theta\circ U_{H,H}\]
on $H$, where
 $(q^\varepsilon_\Theta)_{E^\varepsilon/E_0^\varepsilon}$ was defined in (\ref{toq}). Then by Lemma \ref{sl} and Proposition \ref{cp}, \ref{ps} we immediately get

\begin{lemma}\label{ucl}
We have
\[\lim_{\varepsilon\to0}\|\frac{1}{\varepsilon}
{q^\varepsilon_{\Theta;1}}-q_\Theta\|_\infty=0\]
and the convergence is locally uniform on $\Theta\in\mathbb{R}_{>0}^{|E|}$.
\end{lemma}

Now consider the continuous map
 \begin{equation}\aligned
 \Phi^\varepsilon:\mathbb{R}^{|E|}_{>0} & \to\mathcal{Q}(H)
 :=\{\text{real quadratic forms on $H$}\},\\
  \Theta&\mapsto\frac{1}{\varepsilon}q^\varepsilon_{\Theta;1}.
 \endaligned
 \end{equation}
Then $\Phi^\varepsilon$ converges  locally uniformly to the map
 \begin{equation}\label{PHI}
 \Phi:\mathbb{R}_{>0}^{|E|}\to\mathcal{Q}(H),\quad \Theta\mapsto q_\Theta.
 \end{equation}
Since $\Phi$ is linear and the dimension of $\mathcal{Q}(H)$ is $\frac{N(N+1)}{2}=|E|$, it is a local diffeomorphism. 
As in \cite{C-C} we need the following topological lemma, which follows from the standard Brouwer's Fixed Point Theorem:
\begin{lemma}\label{fixed}
Let $B$ be a closed ball in $\mathbb{R}^n$ and $\varphi:B\to K\subset\mathbb{R}^n$ be a homeomorphism. If we have a family of continuous map $\varphi_\varepsilon:B\to\mathbb{R}^n$ such that $\varphi_\varepsilon$ converges uniformly to $\varphi$ when $\varepsilon\to 0$. Then for any $p\in\mathrm{int}(K)$, we have $p\in\varphi_\varepsilon(B)$ for $\varepsilon$ close to $0$.
\end{lemma}

Now, we are ready to prove
\begin{theorem}\label{eos}
For any finite sequence $0<a_1<a_2\leq a_3\leq\cdots\leq a_N$, there exists a metric $g$ on  $X_N=\Sigma_{\frac{N(N-3)}{2}+1}-\{D_1,\cdots,D_N\}$ such that
\[\lambda_k^\mathcal{D}(X,g)=a_k, \qquad \forall 1\leq k\leq N.\]
\end{theorem}

\begin{proof} By Theorem \ref{CoG}, for the measure
$\mu_0=\sum_{i=1}^N2\pi(N-2)\delta(v_i)$
on the graph $G_N=(V,E)$, there exists an edge weight function $\Theta:E\to\mathbb{R}_{>0}$ such that the eigenvalues of combinatorial Laplacian $\Delta_\Theta$ are exactly $\{a_1,a_2,\cdots,a_N\}$.  By Lemma \ref{ucl}, \ref{fixed}, we have $q_\Theta\in\mathrm{Im}(\Phi^\varepsilon)$ for $\varepsilon$ close to 0. In other words,  there exists $\Theta'=\Theta'(\varepsilon)\in\mathbb{R}^{|E|}_{>0}$ such that \[\frac{1}{\varepsilon}q^{\varepsilon}_{\Theta';1}=q_\Theta.\]
By construction, the eigenvalues of $q^\varepsilon_{\Theta';1}$ are exactly the first $N$ eigenvalues of the Dirichlet Laplacian on the Riemann surface $X^\varepsilon_{\Theta'}$. Let $X^\varepsilon_{\Theta'}$ be $X_N$ equipped with the Riemannian metric $g^{\varepsilon}_{\Theta'}$, then $g={\varepsilon}g^{\varepsilon}_{\Theta'}$ is the metric we want.
\end{proof}
\begin{remark}
Recall that we have denoted the metric on $X^\varepsilon_\Theta$ by $g^\varepsilon_\Theta$, then what we really proved is that the map
\[f_\varepsilon:\mathbb{R}^{|E|}_{>0}\to\mathcal{M}(X_N),\ \Theta\mapsto\varepsilon g^\varepsilon_\Theta,\]
when  restricted to a certain ball centered at $\Theta'(\varepsilon)$, is stable around $\varepsilon g^\varepsilon_{\Theta'}$.
\end{remark}

\section{Prescribing the volume and the eigenvalues}

Now we prove our main theorem. Roughly speak, we embed the surface $X^\varepsilon_{\Theta'}$ into the given manifold $M$ with boundary, such that the boundary of the surface gets embedded into the boundary of $M$, then we construct a Riemannian metric that ``concentrates" near the embedded surface to prescribe the first $N$ eigenvalues of the Dirichlet Laplacian on $M$. Finally, we attach a elongated $n$-dimensional cuboid to the boundary of $M$ to prescribe the volume.

We first adjust  Theorem III.1 in \cite{CV2} to the case of manifold with boundary:

\begin{lemma}\label{At}
Let $(M,g)$ be a compact Riemannian manifold with smooth boundary of dimension $n\geq 3$ and $\Omega_+ \subset M$ be a domain  with piecewise smooth boundary. Consider the Laplacian $\Delta^0$ of mixed boundary conditions on $\Omega_+$ which admits
\begin{center} Neumann on $\partial\Omega_+\cap\mathrm{int}(M)$, \; Dirichlet on $\partial\Omega_+\cap\partial M$. \end{center}
Suppose the eigenvalues of this Laplacian verifies the hypothesis (\ref{hoe}). Then for all $\alpha>0$, there exists a metric $h$ on $M$ with $h|_{\Omega_+}=g$, such that the quadratic forms associated with the Dirichlet Laplacian   $\Delta^{\mathcal{D}}_h$ on $(M,h)$ and the Laplacian $\Delta^0$ have an $N$-spectral difference  $\leq\alpha$, where the eigenspaces are all viewed as subspaces of $L^2(M,g)$.
\end{lemma}

\begin{proof}
As in \cite{CV2}, consider the singular metric
\[g_\varepsilon=\left\{\begin{array}{ll}g, & \text{on }\Omega_+, \\
  \varepsilon\cdot g, & \text{on }\Omega_-=M\setminus\Omega_+.\end{array}\right.\]
Define a quadratic form $q_\varepsilon$ on $H^1_0(M,g_\varepsilon)$ by
 \begin{equation*}
 q_\varepsilon(\varphi)=\int_{\Omega_+}|\nabla\varphi|^2
 \mathrm{d}\mathcal{V}_g+\varepsilon^{\frac{n}{2}-1}\int
 _{\Omega_-}|\nabla\varphi|^2\mathrm{d}\mathcal{V}_g.
 \end{equation*}
By the isometry from $L^2(M,g_\varepsilon)$ to $L^2(M,g)$ defined by $\varphi\mapsto(\varphi|_{\Omega_+},\varepsilon^{\frac{n}{4}
}\varphi|_{\Omega_-})$, one can transport $q_\varepsilon$ to a quadratic form on $L^2(M,g)$ with domain
\[\aligned
D(q_\varepsilon)=\big\{(\varphi_+,\varphi_-)\in H^1(\Omega_+,g)\oplus H^1(\Omega_-,g)\big|\ &\varepsilon^{\frac{n}{4}}\varphi_+|_{\partial
\Omega_+\cap\partial\Omega_-}=\varphi_-|_{\partial\Omega_+\cap\partial\Omega_-},\\   \varphi_+|_{\partial\Omega_+\cap\partial M}& =0,\ \varphi_-|_{\partial\Omega_-\cap\partial M}=0\big\}
\endaligned\]
which is still denoted by $q_\varepsilon$. It follows
 \begin{equation*}
 q_\varepsilon(\varphi_+\oplus\varphi_-)=\int_{\Omega_+}
 |\nabla\varphi_+|^2\mathrm{d}\mathcal{V}_g+\frac{1}{\varepsilon}
 \int_{\Omega_-}|\nabla\varphi_-|^2\mathrm{d}\mathcal{V}_g.
 \end{equation*}

Denote
\[\widetilde{H}^1(\Omega_+)=\{f\in H^1(\Omega_+,g)\big|\ f|_{\partial\Omega_+\cap\partial M}=0\}.\]
Instead of simply using  the harmonic extension operator
\[P_-:H^{\frac{1}{2}}(\partial\Omega_-,g_{\partial M})\to H^1(\Omega_-,g)\]
as in  \cite{CV2}, we need  to use the more complicated operator
\[P^-=P_-\circ\iota\circ\Tr,\]
where
$\Tr$ is the trace operator
\[\Tr:\widetilde{H}^1(\Omega_+)\to H^{\frac{1}{2}}(\partial\Omega_+\cap\partial \Omega_-,g_{\partial\Omega_+\cap\partial\Omega_-}),\]
and
$\iota:\mathrm{Im}(\Tr)\hookrightarrow H^{\frac{1}{2}}(\partial\Omega_-,g_{\partial\Omega_-})$ is the operator
\[\iota (f)=\left\{\begin{array}{ll}f, & \text{ on } \partial\Omega_+\cap\partial\Omega_-, \\
 0, & \text{ on }\partial\Omega_-\cap\partial M.\end{array}\right.\]
The rest of the proof is the same as Theorem III.1 in \cite{CV2}:
with the $q_\varepsilon$-orthogonal decomposition  $D(q_\varepsilon)=\mathcal{H}_0\oplus\mathcal{H}_\infty$ where
\[\mathcal{H}_0=\{(\varphi_+,\varphi_-)\ |\ \Delta\varphi_-=0\},\qquad \mathcal{H}_\infty=\{(0,\varphi_-)\ |\ \varphi_-\in H_0^1(\Omega_-,g)\}.\]
When $\varepsilon$ is small enough, Theorem \ref{ct1} implies $q_\varepsilon$ and $q_\varepsilon|_{\mathcal{H}_0}$ have a small $N$-spectral difference, and Theorem \ref{ct2} and the operator $P^-$ imply that  $q_\varepsilon|_{\mathcal{H}_0}$ and the quadratic form associated to $\Delta^0$ have a small $N$-spectral difference.
The details will be omitted. Last we slightly modify the singular metric $g_\varepsilon$ to get the desired smooth metric $h$.
\end{proof}

Then we introduce a lemma that will be used to prescribe the volume of $M$.

\begin{lemma}\label{van Neu}
Let $(M,g)$ be a compact Riemannian manifold with piecewise smooth boundary $\partial M$. For $p\in \partial M$, denote $b_p(\varepsilon)$ be the $\varepsilon$-geodesic ball in $\partial M$ with center $p$. For any integer $k \ge 1$, let $\lambda_k^\mathcal D (M)$ be the $k$-th eigenvalue of the Dirichlet Laplacian on $M$ and $\lambda^\varepsilon_k(M)$ be the $k$-th eigenvalue of the mixed boundary-valued Laplacian on $M$ with
\begin{center} Dirichlet on $\partial M \setminus b_p(\varepsilon)$ and Neumann on $b_p(\varepsilon)$. \end{center}
Then
\begin{equation}
\lim_{\varepsilon\to 0} \lambda^\varepsilon_k(M)=\lambda_k^\mathcal D (M)
\end{equation}
\end{lemma}

\begin{proof}Firstly, we decompose the space
\[H_\varepsilon^1=\{\ f\in H^1(M)\ |\ f|_{\partial M \setminus b_p(\varepsilon)}=0\}\]
to $\mathcal{H}_0\oplus \mathcal{H}_\infty$ where $\mathcal{H}_0=H_0^1(M)$ and
\[\mathcal{H}_\infty=\mathcal{H}_\infty(\varepsilon)=\{\ f\in H^1_\varepsilon\ |\ \Delta f=0\text{ in }M\}.\]
Then for any $\varphi_0\in \mathcal{H}_0$ and $\varphi_\infty \in \mathcal{H}_\infty$,
\[\int_M \nabla \varphi_0 \cdot \nabla \varphi_\infty \mathrm{d}V_g=0.\]
Let $Q=Q(\varepsilon)$ be the quadratic form on $H^1_\varepsilon$ given by
\[Q(f)=\int_M |\nabla f|^2 \mathrm{d}V_g.\]
We claim that the lower bound of $Q|_{\mathcal{H}_\infty(\varepsilon)}$ tends to $\infty$ as $\varepsilon \to 0$. Suppose that there exists a sequence $\varepsilon_n\to 0$ and a constant $C > 0$ such that for any $n$, there exists $f_n\in H_\infty(\varepsilon_n)$ with
\[ \| f_n \|_{L^2(M)}=1,\qquad Q(f_n) \le C. \]
Then $\{f_n\}$ is a bounded subset in $H^1(M)$. So one can assume that there exists $f\in H^1(M)$ such that $\|f\|_{L^2(M)}=1$ and
\[ f_n \rightharpoonup f\text{ in }H^1(M),\qquad f_n \to f\text{ in }L^2(M).\]
By the compactness of the trace operator $T: H^1(M) \to L^2(\partial M)$, one has $f\in H^1_0(M)=\mathcal{H}_0$. It follows
\[ \int_M |\nabla f|^2 \mathrm{d}V_g=\lim_{n\to \infty} Q(f_n,f)=0,\]
so $f\equiv 0$ which contradicts to the fact $\|f\|_{L^2(M)}=1$. This proves the claim. Now Proposition \ref{van Neu} follows from Theorem \ref{ct1}.
\end{proof}

With all these preparations, one can prove Theorem \ref{mainthm}. 

\begin{proof}[The Proof of Theorem \ref{mainthm}]
Firstly, for any fixed $r>0$ small enough, we construct a family of stable metrics
\[ F^M_r: B \to \mathcal{M}(M), \qquad p \to F^M_r(p)\]
so that there exists a metric in this family with prescribed Dirichlet eigenvalues. For this purpose we
embed the surface $X_N$  constructed in Section \ref{sec4} into $M$ such that the image of $\partial X_N$ is in $\partial M$. Let $\Omega$ be a tubular neighborhood of $X_N$ in $M$ which is diffeomorphic to $X_N\times B^{n-2}$, where $B^{n-2}$ is the unit ball in $\mathbb{R}^{n-2}$. For a family of stable metrics
\[F:B\to\mathcal{M}(X_N),\]
we endow $\Omega$ a family of metrics with an extra parameter $r$,
\[F_r:B\to\mathcal{M}(\Omega),\qquad p\to F(p) + r^2g_{\mathrm{Eucl}}.\]
Consider the mixed boundary-valued Laplacian $\Delta^0_\Omega$ on $\Omega$ with
\[\text{Dirichlet on }\partial X_N\times B^{n-2},\qquad\text{Neumann on }(X_N\setminus\partial X_N)\times\partial B^{n-2},\]
i.e. the Friedrich extension of the quadratic form $q(f)=\int_\Omega|\nabla f|^2$ with domain
\[\mathcal{D}=\{f\in H^1(\Omega)|\ f=0\text{ on }\partial X_N\times B^{n-2}\}.\]
Then we have a $q$-orthogonal decomposition $\mathcal{D}=\mathcal{H}_\infty\oplus\mathcal{H}_0$, where
\begin{equation*}
\begin{aligned}
&\mathcal{H}_\infty=\{f\in\mathcal{D}|\ \int_{B^{n-2}}f(x,y)\mathrm{d}y=0,\ \forall x\in X_N\}, \\
&\mathcal{H}_0=\{f\in\mathcal{D}|\ f(x,y)\text{ is independent of } y\}.
\end{aligned}
\end{equation*}
By the fact that the first nonzero Neumann eigenvalue of $(B^{n-2},r^2g_{\mathrm{Eucl.}})$ tends to infinity as $r$ goes to zero and by Theorem \ref{ct1},  $F_r$ is spectrally near to $F$ when $r$ is small enough. Thus one can prescribe the eigenvalues of $\Delta^0_\Omega$  to be the given sequence in Theorem \ref{mainthm}. Next, Lemma \ref{At} allows us to get the desired stable metrics $F^M_r$ on $M$ and thus, one can prescribe the Dirichlet eigenvalues on $M$. The detail is similar to the proof of the main Theorem in \cite{CV2} and thus is omitted.

To prescribe the volume simultaneously, we fix $r$ small enough such that
\[ \mathrm{Vol}(M,F^M_r(p)) < \frac{V}{2}.\]
Then attach a $n$-dimensional cuboid $R=[0,a]^{n-1} \times [0,b]$ to $M$ along a small boundary $(n -1)$-dimensional ball $I$ with radius $c$, as illustrated below

\includegraphics[scale=0.5]{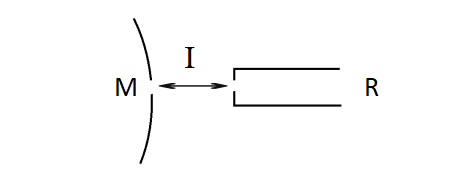}\\
where $a,b,c$ will be carefully chosen below, with
\[\mathrm{Vol}(M,F^M_r(p))+\mathrm{Vol}(R)=V.\]
Denote the resulting manifold by $M_R=M \cup_I R$. Decompose
\[ H^1_0(M_R)=\mathcal{H}_0 \oplus \mathcal{H}_\infty,\]
where $\mathcal{H}_0=H^1_0(M)$ and
\[\mathcal{H}_\infty=\mathcal{H}_\infty(I)=\{ f \in H^1_0(M_R)\ |\ \Delta f=0 \text{ in } M\}.\]
By the proof of Lemma \ref{van Neu}, for any $T>0$, one has
\[ \int_M |\nabla f|^2 \mathrm{d} V_{F^M_r(p)} \geq T \int_M f^2 \mathrm{d} V_{F^M_r(p)},\qquad \forall f \in \mathcal{H}_\infty\]
for $c$ small enough.

Let $\lambda_1(c)$ be the first eigenvalues of the mixed boundary-valued Laplacian on $R$ with
\begin{center} Dirichlet on $\partial R\!\setminus\! I\;$ and $\;$ Neumann on $I$. \end{center}
By Lemma \ref{van Neu}, for  $c$  small enough one has
\[\lambda_1(c) \ge \frac{\lambda^\mathcal D_1(R)}{2}=\frac{\pi^2}{2}(\frac{n-1}{a^2}+\frac{1}{b^2}).\]
Moreover, one can make $\lambda^\mathcal{D}_1(R) \ge 2T$ by carefully choosing $a$ and $b$. So that
\[ \int_{M_R} |\nabla f|^2 \mathrm{d} V =\int _M |\nabla f|^2 \mathrm{d} V_{F^M_r(p)} + \int_R |\nabla f|^2 \mathrm{d} V \ge T \int_{M_R} f^2 \mathrm{d} V,\ \forall f\in \mathcal{H}_\infty,\]
after carefully choosing $a$, $b$, $c$. By Theorem \ref{ct1} and the fact that $M_R$ is diffeomorphic to $M$, when $T$ is big enough, one gets a family of stable metrics on $M$ such that the volume of $M$ under any metric in this family is $V$. One also can slightly change the metrics on $M_R$ such that the new metrics are smooth with the same area and are still stable. This completes the proof.
\end{proof}

\section{The proof of Theorem \ref{ct1}}

\par Before giving the proof of Theorem \ref{ct1}, we  provide a counterexample of Proposition I.5 in \cite{CV2}: 

\begin{example}\label{couExmp}
Let $\mathcal{H}$ be the Euclidean space $\mathbb{R}^4$, and $Q=Q_b$ be the quadratic form on $\mathbb{R}^4$ with matrix representation
\[\begin{pmatrix}1&0&0&-b \\ 0&1& -b& 0  \\ 0& -b& b^2& 0 \\ -b& 0 &0 &b^2\end{pmatrix}\]
under the canonical basis.
For simplify, we denote
\[\varphi_\infty^1=(1,0,0,0),\qquad \varphi_\infty^2=(0,1,0,0),\]
and
\[\varphi_0^1=(0,b,1,0),\qquad \varphi_0^2=(b,0,0,1).\]
Then
 \[Q(\varphi^i_\infty,\varphi^j_0)=0,\qquad i=1,2;\ j=1,2\]
and thus we have a $Q$-orthogonal decomposition $\mathcal{H} =\mathcal{H}_\infty\oplus\mathcal{H}_0$, where
\[\mathcal{H}_\infty=\mathbb{R}  \varphi_\infty^1\oplus\mathbb{R} \varphi_\infty^2, \qquad \mathcal{H}_0=\mathbb{R} \varphi_0^1\oplus\mathbb{R}\varphi_0^2.\]
For any $x \in \mathcal{H}_\infty$, one has $Q(x) \ge \|x\|^2$. However, if we let
\[x=\varphi_0^1\wedge\varphi^1_\infty+\varphi^2_\infty \wedge \varphi_0^2+b\varphi^1_\infty \wedge \varphi^2_\infty,\]
then direct computations yield
\[Q^{\wedge^2}(x)=4, \qquad |x|^2_{\wedge^2}=2b^2+4.\]
Thus for the inequality $Q^{\wedge^2}(x) \ge C_2 |x|^2_{\wedge^2}$ to be true, the constant $C_2$ should depend on the choice of $b$, which contradicts to the conclusion of Proposition I.5 in \cite{CV2}.
\end{example}

\begin{proof}[The Proof of Theorem \ref{ct1}]

According to Lemma \ref{crit} (where $A_0=A_1=\mathrm{Id}$ in the current setting), it is enough to prove that as $C \to \infty$,
\begin{enumerate}
  \item $\|B\| \to 0$ , \
  \item $\max_{1\leq k\leq N}|\lambda_k(q_1)-\lambda_k(q_0)| \to 0$.
\end{enumerate}

Let $\{\psi_i\}_{i\geq1}$ be an orthonormal set of eigenfunctions corresponding to $\{\mu_i\}_{i \ge 1}$, and likewise let $\{\lambda_i\}_{i\geq1}$, $\{\varphi_i\}_{i\geq1}$ be eigenvalues and corresponding orthonormal set of eigenfunctions for $Q$. Let
\[E_0^m=\mathrm{span}\{\psi_1,\cdots,\psi_m\},\qquad  E_1^m=\mathrm{span}\{\varphi_1,\cdots,\varphi_m\},\]
where $1 \le m \le N$ is an integer  such that $\mu_m<\mu_{m+1}$,
and let $P_m$ be the orthogonal projection to $E_0^m$.

We first claim that $P_m:E_1^m\to E_0^m$ is surjective
for $C$ large enough.
In fact,
suppose $P_m$ is not surjective, then there exists $\varphi\in E_1^m$ with $|\varphi|=1$ such that $P_m\varphi=0$. Let $\varphi=\varphi^0+\varphi^\infty\in\mathcal{H}_0\oplus\mathcal{H}_\infty$, then
  \begin{equation*}
  M\geq\mu_m\geq\lambda_m\geq Q(\varphi)=Q(\varphi^0)+Q(\varphi^\infty)\geq C|\varphi^\infty|^2.
  \end{equation*}
So
  \begin{equation}\label{phi-infty}
    |\varphi^\infty|\leq\sqrt{\frac{M}{C}}\; \text{ and }\; |\varphi^0|\geq 1-\sqrt{\frac{M}{C}}.
  \end{equation}
Since $P_m\varphi=0$, we have $|P_m\varphi^0|=|P_m\varphi^\infty|\leq|\varphi^\infty|\leq\sqrt{\frac{M}{C}}$
and thus
\[|(I-P_m)\varphi^0|\geq|\varphi^0|-|P_m\varphi^0|\geq 1-2\sqrt{\frac{M}{C}}.\]
Since $(I-P_m)\varphi^0 \in \mathcal H_0$ and $P_m((I-P_m)\varphi^0 )=0$, we see
\[Q((I-P_m)\varphi^0)\geq(\mu_m+\tilde \delta)
 (1-2\sqrt{\frac{M}{C}})^2\]
 and thus
 \begin{equation*}
 \mu_m\geq Q(\varphi^0)=Q(P\varphi^0)+Q((I-P)\varphi^0)\geq(\mu_m+\tilde \delta)
 (1-2\sqrt{\frac{M}{C}})^2.
 \end{equation*}
It follows
 \begin{equation*}
 (1-2\sqrt{\frac{M}{C}})^2\leq\frac{\mu_m}{\mu_m+\tilde \delta}
 \leq\frac{M}{M+\tilde \delta}.
 \end{equation*}
So we get a contradiction if $C$ is large. More precisely, if we take
\[C>4M(1-\sqrt{\frac{M}{M+\tilde \delta}})^{-2},\]
then  $P_m:E_1^m\to E_0^m$ is surjective, which proves the claim.

As a consequence, for such a $C$, there exists \[B\in\mathcal{L}(E_0^N,{(E_0^N)}^\perp)\]
such that $E_1^N$ is the graph of $B$.

Next we prove (2) first.
Let $\varphi_i=\varphi_i^0+\varphi_i^\infty\in\mathcal{H}_0\oplus
\mathcal{H}_\infty$. By (\ref{phi-infty}),
 \begin{equation*}
 |\langle\varphi_i^0,\varphi_j^0\rangle-\langle\varphi_i,\varphi_j\rangle|\leq 2\sqrt{\frac{M}{C}}\ \text{for all }1\leq i,j\leq k,
 \end{equation*}
so if we denote $E_k^\ast:=\mathrm{span}\{\varphi_1^0,\cdots,\varphi_k^0\}$, then $\dim(E_k^\ast)=k$
for $C$ large enough. Let
\[T:E_k^*\to\mathcal{H}_\infty\]
be the linear mapping   that maps $\varphi_i^0$ to $\varphi_i^\infty$ for all $1\leq i\leq k$. To control the operator norm $\|T\|$, we observe that
 \begin{equation*}
 |T(\sum_{i=1}^k\alpha_i\varphi_i^0)|=|\sum_{i=1}^k
 \alpha_i\varphi_i^\infty|\leq(\sum_{i=1}^k|\alpha_i|)
 \sqrt{\frac{M}{C}}
 \end{equation*}
and
 \begin{equation*}
 \begin{aligned}
 |\sum_{i=1}^k\alpha_i\varphi_i^0|^2
 =&    |\sum_{1\leq i,j\leq k}\alpha_i\alpha_j\langle\varphi_i^0,\varphi_j^0\rangle|\\
 \geq& \sum_{i=1}^k|\alpha_i|^2(1-\sqrt{\frac{M}{C}})^2-(
 \sum_{i\neq j}|\alpha_i\alpha_j|)2\sqrt{\frac{M}{C}}\\
 \geq& \sum_{i=1}^k|\alpha_i|^2(1-\sqrt{\frac{M}{C}})^2-
 (\sum_{i\neq j}|\alpha_i|^2+|\alpha_j|^2)\sqrt{\frac{M}{C}}\\
 =&   \sum_{i=1}^k|\alpha_i|^2(1-2k\sqrt{\frac{M}{C}}+\frac{M}{C})\\
 \geq& (\sum_{i=1}^k|\alpha_i|)^2\frac{1}{k}(1-2k\sqrt{\frac{M}{C}}
 +\frac{M}{C}).
 \end{aligned}
 \end{equation*}
It follows
\[\|T\|\leq\sqrt{\frac{M}{C}}\sqrt{k}(1-2k\sqrt{\frac{M}{C}}+\frac{M}{C})^{-\frac{1}{2}}.\]
On the other hand, by the min-max principle of eigenvalues,
 \begin{equation*}
 \begin{aligned}
 \lambda_k&=\max_{0\neq\psi\in E_k^*}\frac{Q(\psi+T\psi)}{|\psi+T\psi|^2}\geq\max_{0\neq\psi\in E_k^*}\frac{Q(\psi)}{|\psi+T\psi|^2}\geq\frac{1}{(1+\|T\|)^2}\max_{0\neq\psi\in E_k^\ast}\frac{Q(\psi)}{|\psi|^2}\\
 &\geq\frac{\mu_k}{(1+\|T\|)^2},
 \end{aligned}
 \end{equation*}
and thus
\[0\leq\mu_k-\lambda_k\leq\lambda_k[(1+\|T\|)^2-1]\leq M[(1+\|T\|)^2-1].\]
So $\mu_k-\lambda_k\to0$ as $C\to\infty$. This finishes the proof of (2).

Finally we prove (1), i.e. $\|B\|_{\mathcal{L}(E_0^N,{E_0^N}^\perp)} \to 0$ as $C\to\infty$. Observe that
\[B\circ P_Nx=x-P_Nx, \quad \forall x \in E_1^N,\]
so $\|B\|_{\mathcal{L}(E_0^N,{E_0^N}^\perp)}$ is close to zero if $\|I-P_N\|_{\mathcal{L}(E_1^N,\mathcal{H})}$ is close to zero. For simplicity, we first assume
\[a=\mu_1=\cdots=\mu_i<\mu_{i+1}=\cdots=\mu_N=b.\]
Now $\tilde \delta=\min(\delta,b-a)$. For any $1\leq k \leq i$,
 \begin{equation}\label{mu1}
 \begin{aligned}
 a\geq& Q(\varphi_k)\geq Q(\varphi_k^0)=Q(P_N\varphi_k^0)+Q((I-P_N)\varphi_k^0)\\
 \geq&a|P_N\varphi_k^0|^2+(b+\tilde \delta)|(I-P_N)\varphi_k^0|^2.
 \end{aligned}
 \end{equation}
Since
\[|P_N\varphi_k^0|^2+|(I-P_N)\varphi_k^0|^2 = |\varphi_k^0|^2\geq(1-\sqrt{\frac{M}{C}})^2\geq1-2\sqrt{\frac{M}{C}},\]
the following equation holds
\[|P_N\varphi_k^0|^2\geq1-|(I-P_N)\varphi_k^0|^2-2\sqrt{\frac{M}{C}}.\]
Combine with (\ref{mu1}),
\[|(I-P_N)\varphi_k^0|^2\leq2\sqrt{\frac{M}{C}}\frac{a}{b-a+\tilde \delta} { \leq 2\sqrt{\frac{M}{C}}\frac{M}{\tilde \delta}}.\]
So $|(I-P_N)\varphi_k^0|\to 0$ as $C\to\infty$ for all $1 \leq k \leq i$.

Similarly, one can get
 \begin{equation*}
 |(I-P_i)\varphi_k^0|^2\leq 2\sqrt{\frac{M}{C}}\frac{M}{\tilde \delta}
 \end{equation*}
for all $1\leq k \leq i$. By careful choosing $\{\psi_k\}_{k=1}^i$, one can guarantee that $|\varphi^0_k-\psi_k|\to 0$ as $C\to\infty$ for all $1\leq k \leq i$. Then for all $1 \leq k \leq i$, $i+1 \leq l \leq N$,
 \begin{equation*}
 \begin{aligned}
 |\langle\varphi_l^0,\psi_k\rangle|\leq&|\langle\varphi_l^0,
 \varphi_k^0\rangle|+|\langle\varphi_l^0,\psi_k-\varphi_k^0\rangle|\\
 \leq&2\sqrt{\frac{M}{C}}+|\langle\varphi_l^0,\psi_k-\varphi_k^0\rangle|\to0,\text{ as}\ C\to\infty.
 \end{aligned}
 \end{equation*}
So if we write $P_N\varphi_l^0=\alpha_l+\varphi'_l$ with $\alpha_l\in E^i_0$ and $\varphi'_l\perp E^i_0$, then $|\alpha_l|\to0$ as $C\to\infty$. By our assumptions,
 \begin{equation}\label{b}
 \begin{aligned}
 b=&\mu_l\geq\lambda_l=Q(\varphi_l)\geq Q(\varphi_l^0)=Q(P_N\varphi_l^0)+Q((I-P_N)\varphi_l^0)\\
 \geq&b|\varphi_l'|^2+(b+\tilde \delta)|(I-P_N)\varphi_i^0|^2.
 \end{aligned}
 \end{equation}
And
\[|\varphi_l^0|^2=|\alpha_l|^2+|\varphi_l'|^2+|(I-P_N)\varphi_l^0|^2\geq(1-\sqrt{\frac{M}{C}})^2\geq1-2\sqrt{\frac{M}{C}}\]
implies
\begin{equation}\label{psi-i'}
|\varphi_l'|^2\geq1-2\sqrt{\frac{M}{C}}-|\alpha_l|^2-|(I-P_N)\varphi_l^0|^2.
\end{equation}
Then combining (\ref{psi-i'}) with (\ref{b}),
\[|(I-P_N)\varphi_l^0|^2\leq\frac{b}{\tilde \delta}(2\sqrt{\frac{M}{C}}+|\alpha_i|^2)\]
and thus $|(I-P_N)\varphi_l^0|\to 0$ as $C\to\infty$ for all $i+1 \leq l \leq N$. Combing all discussions above and the fact
\[ |(I-P_N)\varphi^\infty_j| \leq |\varphi^\infty_j| \leq \sqrt{\frac{M}{C}},\qquad \forall 1 \leq j \leq N, \]
we have $\|I-P_N\|_{\mathcal{L}(E_1^N,\mathcal{H})}\to 0$ as $C\to\infty$, which implies $\|B\|_{\mathcal{L}(E_0^N,{E_0^N}^\perp)}\to 0$ as $C\to\infty$.
This proves (1) in the simple case.

To prove (1) for the general case
\[\mu_1 =\cdots =\mu_{i_1}< \mu_{i_1+1}=\cdots=\mu_{i_2}<\cdots \leq\mu_N< \mu_{N+1}\leq M,\]
one just apply the same idea several times, namely, project onto the eigenspace $E_{\lambda_i}$ for each distinguished eigenvalue below $\mu_{N+1}$.

\end{proof}


\begin{thebibliography}{99}


\bibitem{Arn}
V. Arnol'd: Modes and quasimodes. \emph{Funktsional'nyi Analiz i ego Prilozheniya}, 1972, 6(2): 12-20.
\bibitem{Ber}
A. Berdnikov: Bounds on multiplicities of Laplace operator eigenvalues on surfaces. \emph{Journal of Spectral Theory}, 2018(2): 541-554.
\bibitem{PB}
P. Buser: Geometry and spectra of compact Riemann surfaces. \emph{Springer Science $\&$ Business Media}, 2010.
\bibitem{C-C}
B. Colbois and Y. Colin de Verdi\`{e}re: Sur la multiplicit\'{e} de la premi\`{e}re valeur propre d'une surface de Riemann \`{a} courbure constante. \emph{Commentarii Mathematici Helvetici}, 1988, 63: 194-208.
\bibitem{CV2}
Y. Colin de Verdi\`{e}re: Sur la multiplicit\'{e} de la premi\`{e}re valeur propre non nulle du Laplacien. \emph{Commentarii Mathematici Helvetici}, 1986, 61(1): 254-270.
\bibitem{CV}
Y. Colin de Verdi\`{e}re: Construction de laplaciens dont une pertie finie du spectre est donn\'{e}e. \emph{Annales scientifiques de l'\'{e}cole normale sup\'{e}rieure}, 1987, 20(4): 599-615.
\bibitem{CV1}
Y. Colin de Verdi\`{e}re: Sur une hypoth\`{e}se de transversalit\'{e} d'Arnold. \emph{Commentarii Mathematici Helvetici}, 1988, 63: 184-193.
\bibitem{MD}
M. Dahl: Prescribing eigenvalues of the Dirac operator. \emph{Manuscripta Mathematica}, 2005, 118(2): 191-199.
\bibitem{DPRS}
J. Dodziuk, T. Pignataro, B. Randol and D. Sullivan: Estimating small eigenvalues of Riemann surfaces. The legacy of Sonya-Kovalevskaya (Cambridge, Mass., and Amherst, Mass., 1985), \emph{Contemp. Math}, 1987, 64: 93-121.
\bibitem{HKP}
A. Hassannezhad, G. Kokarev and I. Polterovich: Eigenvalue inequalities on Riemannian manifolds with a lower Ricci curvature bound. \emph{Journal of Spectral Theory}, 2016, 6(4): 807-835.
\bibitem{He}
X. He: Prescription of finite Dirichlet eigenvalues and area on surface with nonempty boundary. Prepint.
\bibitem{PJ}
P. Jammes: Prescription de la multiplicit\'{e} des valeurs propres du laplacien de Hodge-de Rham. \emph{Commentarii Mathematici Helvetici}, 2011, 86(4): 967-984.
\bibitem{PJ12}
P. Jammes: Sur la multiplicit\'{e} des valeurs propres du laplacien de Witten. \emph{Transactions of the American Mathematical Society}, 2012, 364(6): 2825-2845.
\bibitem{PJ14}
P. Jammes: Prescription du spectre de Steklov dans une classe conforme. \emph{Analysis} \& \emph{PDE}, 2014, 7(3): 529-550.
\bibitem{Kac}
M. Kac: Can one hear the shape of a drum? \emph{The American Mathematical Monthly}, 1966, 73(4P2): 1-23.
\bibitem{JL96}
J. Lohkamp: Discontinuity of geometric expansions. \emph{Commentarii Mathematici Helvetici}, 1996, 71(1): 213-228.
\bibitem{RB}
B. Randol: Cylinders in Riemann surfaces. \emph{Commentarii Mathematici Helvetici}, 1979, 54: 1-5.
\bibitem{RS}
M. Reed and B. Simon: Methods of modern mathematical physics: Functional analysis; Rev. ed. \emph{Academic press}, 1980.

\end{thebibliography}
\end{document}